\documentclass[12pt]{article}

\usepackage{amsfonts,amsmath,amssymb,mathrsfs,mathtools,url}
\usepackage{amsthm,latexsym}
\usepackage{tikz,color}

\newtheorem{theorem}{Theorem}

\newtheorem{lemma}{Lemma}

\DeclareMathOperator {\gp} {gp}

\let\deg\relax
\DeclareMathOperator {\deg} {deg}

\def\cp{\,\square\,}

\vspace{8mm}
\title{The general position number of the Cartesian product of two trees}

\author{Jing Tian$\/^{a}$, Kexiang Xu$\/^{a}$, Sandi Klav\v{z}ar$^{b, c, d}$ \\\\
 $^{a}$ \small  College of Science, Nanjing University of
 Aeronautics \& Astronautics,\\
 \small Nanjing, Jiangsu 210016, PR China\\
$^{b}$ \small Faculty of Mathematics and Physics, University of Ljubljana, Slovenia\\
$^{c}$ \small Faculty of Natural Sciences and Mathematics, University of Maribor, Slovenia\\
$^{d}$ \small Institute of Mathematics, Physics and Mechanics, Ljubljana, Slovenia \\
\small {\tt jingtian526@126.com} (J.\ Tian)\\
\small {\tt kexxu1221@126.com} (K.\ Xu)\\
\small {\tt sandi.klavzar@fmf.uni-lj.si} (S.\ Klav\v{z}ar)
}
\date{}

\begin{document}

\maketitle

\begin{abstract}
 The general position number of a connected graph is the cardinality of a largest set of vertices such that no three  pairwise-distinct vertices from the set lie on a common shortest path. In this paper it is proved that the general position number is additive on the Cartesian product of two trees.
\end{abstract}


\medskip\noindent
\textbf{AMS Math.\ Subj.\ Class.\ (2020)}: 05C05, 05C12, 05C35


\section{Introduction}
\label{sec:intro}

Let $d_G(x,y)$ denote, as usual, the number of edges on a shortest $x,y$-path in $G$. A set $S$ of vertices of a connected graph $G$ is a {\em general position set} if $d_G(x,y) \ne d_G(x,z) + d_G(z,y)$ holds for every $\{x,y,z\}\in \binom{S}{3}$.  The {\em general position number} $\gp(G)$ of $G$ is the cardinality of a largest general position set in $G$. Such a set is briefly called a {\em gp-set} of $G$.

Before the general position number was introduced in~\cite{PM}, an equivalent concept was proposed in~\cite{ullas-2016}. Much earlier, however, the general position problem has been studied  by K\"orner~\cite{korner-1995} in the special case of hypercubes. Following~\cite{PM}, the graph theory general position problem has been investigated in~\cite{BSA, MGS, Sk, SKB, PMS, patkos-2019+, JT}.

The {\em Cartesian product} $G\cp H$ of vertex-disjoint graphs $G$ and $H$ is the graph with vertex set $V(G) \times V(H)$, vertices $(g,h)$ and $(g',h')$ being adjacent if either $g=g'$ and $hh'\in E(H)$, or $h=h'$ and $gg'\in E(G)$. In this paper we are interested in $\gp(G\cp H)$, a problem earlier studied in~\cite{MGS, SKB, PMS, JT}. More precisely, we are interested in Cartesian products of two (finite) trees. (For some of the other investigations of the Cartesian product of trees see~\cite{balak-2016, shiu-2018, wood-2011}.) An important reason for this interest is the fact that the general position number of products of paths is far from being trivial. First, denoting with $P_\infty$ the two-way infinite path, one of the main results from~\cite{PMS} asserts that $\gp(P_\infty \cp P_\infty) = 4$. Denoting further with $G^n$ the $n$-fold Cartesian product of $G$, it was demonstrated in the same paper that $10\le \gp(P_\infty^3) \le 16$. The lower bound $10$ was improved to $14$ in~\cite{SKB}. Very recently, these results were superseded in~\cite{KR} by proving that if $n$ is an arbitrary positive integer,  then $\gp(P_\infty^n) = 2^{2^{n-1}}$. Denoting with $n(G)$ the order of a graph $G$, in this paper we prove:

\begin{theorem}
\label{thm:main}
If $T$ and $T^{*}$ are trees with $\min\{n(T),n(T^*)\}\geq 3$, then
$$\gp(T\cp T^{*}) = \gp(T) + \gp(T^{*})\,.$$
\end{theorem}

\noindent
Theorem~\ref{thm:main} widely extends the above mentioned result $\gp(P_\infty \cp P_\infty) = 4$. Further,  the equality $\gp(P_\infty^n) = 2^{2^{n-1}}$ shows that Theorem~\ref{thm:main} has no obvious (inductive) extension to Cartesian products of more than two trees. Hence, to determine the general position number of such products remains a challenging problem.

In the next section we give further definitions, recall known results needed, and prove several auxiliary new results. Then, in Section~\ref{sec:proof}, we prove Theorem~\ref{thm:main}.

\section{Preliminaries}
\label{sec:prelim}

 Let $T$ be a tree. The set of leaves of $T$ will be denoted by $L(T)$, and let $\ell(T) = |L(T)|$. If $u$ and $v$ are vertices of $T$ with $\deg(u) \ge 2$ and $\deg(v) = 1$, then the unique $u,v$-path is a \textit{branching path} of $T$. If $u$ is not a leaf of $T$, then there are exactly $\ell(T)$ branching paths starting from $u$; we say that the $u$ is the  \textit{root} of these branching paths and that the degree $1$ vertex of a branching path $P$ is the \textit{leaf of} $P$.

\begin{lemma}\label{tree}{\rm(\cite{PM})}
If $T$ is a tree, then $\gp(T)=\ell(T)$.
\end{lemma}

We next describe which vertices of a tree lie in some gp-set of the tree.

\begin{lemma}\label{T-non-leaf}
A non-leaf vertex $u$ in a tree $T$ belongs to a gp-set of $T$ if and only if $T-u$ has exactly two components and at least one of them is a path.
\end{lemma}

\begin{proof}
First, let $R$ be a gp-set of $T$ containing the non-leaf vertex $u$. Suppose that $T-u$ has at least three components, say $T_1,T_2$ and $T_3$. Since $R$ is a gp-set containing $u$, $R$ intersects with at most one of $T_1$, $T_2$ and $T_3$. Assume without loss of generality that $R\cap V(T_2)=\emptyset$ and $R\cap V(T_3)=\emptyset$. Choose vertices $v$ and $w$ in $T$ such that $v\in V(T_2)$ and $w\in V(T_3)$. Then  $(R-\{u\})\cup\{v,w\}$ is a larger gp-set than $R$ in $T$, a contradiction. Hence $T-u$ has exactly two components, say $T_1$ and $T_2$. Now suppose that neither  $T_1$ nor $T_2$ is a path. Then as above, we have  $R\cap V(T_1)=\emptyset$ or $R\cap V(T_2)=\emptyset$. By symmetry, we assume that $R\cap V(T_2)=\emptyset$. Since $T_2$ is not a path, there are at least two leaves $x_1$ and $x_2$ in $T_2$.
Then the set $(R-\{u\})\cup\{x_1,x_2\}$ is a larger gp-set than $R$, again, in $T$. Therefore, at least one of $T_1$ and $T_2$ is a path.

Conversely, we observe that $u$ is a non-leaf vertex on a pendant path in $T$. Then $u$ belongs to a gp-set in $T$.
\end{proof}

In $G\cp H$, if $h\in V(H)$, then the subgraph of $G\cp H$ induced by the vertices $(g,h)$, $g\in V(G)$, is a {\em $G$-layer}, denoted with  $G^h$. Analogously $H$-layers $\prescript{g}{}H$ are defined. $G$-layers and $H$-layers are isomorphic to $G$ and to $H$, respectively. The distance function in Cartesian products is additive, that is, if $(g_{1},h_{1}), (g_{2},h_{2})\in V(G\cp H)$, then
\begin{equation}
\label{eq1}
d_{G\cp H}((g_{1},h_{1}), (g_{2},h_{2})) = d_{G}(g_{1},g_{2})+d_{H}(h_{1},h_{2}).
\end{equation}
If $u,v\in V(G)$, then the {\em interval} $I_G(u,v)$ between $u$ and $v$ in $G$ is the set of all vertices lying on shortest $u,v$-paths, that is,
$$I_G(u,v) = \{w:\ d_G(u,v) = d_G(u,w) + d_G(w,u)\}\,.$$
In what follows, the notations $d_{G}(u,v)$ and $I_G(u,v)$ may be simplified to $d(u,v)$ and $I(u,v)$ if $G$ will be  clear from the context.
Equality~\eqref{eq1} implies that intervals in Cartesian products have the following nice structure, cf.~\cite[Proposition 12.4]{WI}.

\begin{lemma}\label{geo-interval}
If $G$ and $H$ are connected graphs and $(g_{1},h_{1}), (g_{2},h_{2})\in V(G\cp H)$, then
$$I_{G\cp H}((g_{1},h_{1}), (g_{2},h_{2})) = I_{G}(g_{1},g_{2}) \times I_{H}(h_{1},h_{2})\,.$$
\end{lemma}

Equality (\ref{eq1}) also easily implies the following fact (also proved in \cite{JT}).

\begin{lemma}
\label{lem:lay-v}
Let $G$ and $H$ be connected graphs and $R$ a general position set of $G\cp H$. If $u=(g,h)\in R$, then $V(\prescript{g}{}{H})\cap R=\{u\}$ or $V(G^{h})\cap R=\{u\}$.
\end{lemma}

For finite paths the already mentioned result $\gp(P_\infty \cp P_\infty) = 4$ reduces to:

\begin{lemma}
\label{p-path}{\rm(\cite{PMS})}
If $n_1, n_2\ge 2$, then
$$\gp(P_{n_1}\cp P_{n_2})= \left\{
\begin{array}{ll}
4; & \min\{n_1,n_2\}\geq 3, \\
\\
3;& {\rm otherwise}\,.
\end{array}
\right.$$
\end{lemma}

To conclude the preliminaries we construct special maximal (with respect to inclusion) general position sets in products of trees.

\begin{lemma}\label{lower-gp} Let $T$ and $T^*$ be two trees with $\min\{n(T),n(T^*)\}\geq 3$, $v_{i}\in V(T)\setminus L(T)$, and $v_{j}^{*}\in V(T^*)\setminus L(T^*)$. Then $(L(T)\times\{v_{j}^{*}\})\cup(\{v_{i}\}\times L(T^{*}))$ is a maximal general position set of $T\cp T^{*}$.
\end{lemma}

\begin{proof}
Set $R=(L(T)\times\{v_{j}^{*}\})\cup(\{v_{i}\}\times L(T^{*}))$ and let $V_0=\{u,v,w\}\subseteq R$. We first consider the case when $V_0\subseteq L(T)\times\{v_{j}^{*}\}$ or $V_0\subseteq \{v_{i}\}\times L(T^{*})$. By symmetry, assume that $V_0\subseteq L(T)\times\{v_{j}^{*}\}$. Then each vertex of $V_0$ is corresponding to a leaf of $L(T)$ in the layer $T^{v_j^*}\cong T$. Therefore $u,v,w$ do not lie on a common geodesic in $T\cp T^{*}$.

In the following, without loss of generality, we can assume that $u,w\in L(T)\times\{v_{j}^{*}\}$ with $u=(v_{k},v_{j}^{*})$, $w=(v_s,v_{j}^{*})$ and $v=(v_{i},v_{\ell}^{*})\in\{v_{i}\}\times L(T^{*})$.  By Equality (\ref{eq1}), we have $d(u,v)=d_{T}(v_{k},v_{i})+d_{T^{*}}(v_{j}^{*},v_{\ell}^{*})$ and $d(u,w)=d_{T}(v_{k},v_s)$, $d(w,v)=d_{T}(v_s,v_{i})+d_{T^{*}}(v_{j}^{*},v_{\ell}^{*})$. Note that $v_k$, $v_s$ are two distinct vertices in $L(T)$ of $T$ and $v_i\in V(T)\setminus L(T)$. Then $d_T(v_k,v_i)<d_T(v_k,v_s)+d_T(v_s,v_i)$ whenever $v_i$ lies on the $v_k,v_s$-geodesic or outside $v_k,v_s$-geodesic of $T$. This implies that $d(u,v)<d(u,w)+d(w,v)$ in $T\cp T^{*}$. Therefore $w$ does not lie on the $u,v$-geodesic in $T\cp T^{*}$. Analogously, neither $u$ lies on the $v,w$-geodesic nor $v$ lies on the $u,w$-geodesic of $T\cp T^{*}$. Thus $u,v,w$ do not lie on a common geodesic in $T\cp T^{*}$, which implies that $R$ is a general position set in $T\cp T^{*}$.

Next we prove the maximality of $(L(T)\times\{v_{j}^{*}\})\cup(\{v_{i}\}\times L(T^{*}))$ as a general position set in $T\cp T^{*}$. Otherwise, there is a general position set $R^{\prime}$ in $T\cp T^*$ of order greater than $\ell(T)+\ell(T^*)$ such that $R\subset R^{\prime}$. Then there exists a vertex $z\in R^{\prime}\backslash R$,  say  $z=(v_{p},v_{q}^{*})$. If $p=i$, then there exist two vertices $(v_i,v_s^*),(v_i,v_t^*)\in R$ such that  $z\in I_{T\cp T^{*}}((v_i,v_s^*),(v_i,v_t^*))$ (since $\prescript{v_i}{}{T^{*}}\cong T^{*}$). This is a contradiction showing that $p\neq i$. Similarly, we have $q\neq j$. Now we consider the positions of $v_p$ in $T$ and $v_q^{*}$ in $T^*$. Suppose first that $v_{p}\in L(T)$, $v_{q}^{*}\in L(T^{*})$. Then there are two vertices $(v_{p},v_{j}^{*}),(v_{i},v_{q}^{*})$ in $R$ such that  $z\in I_{T\cp T^{*}}((v_{p},v_{j}^{*}),(v_{i},v_{q}^{*}))$, contracting that $R\cup \{z\}$ is a general position set of $T\cp T^{*}$. If $v_{p}\in L(T)$ and $v_{q}^{*}\notin L(T^{*})$, then we select a vertex $v_{q^{\prime}}^*\in L(T^*)$ such that $v_{q^{\prime}}^*$ is closer to the leaf of the corresponding branching path than $v_q^*$ in $T^*$.  Then $z\in I_{T\cp T^{*}}((v_{p},v_{j}^{*}),(v_{i},v_{q^{\prime}}^{*}))$,  a contradiction. Similarly, $v_{p}\notin L(T)$ and $v_{q}^{*}\in L(T^{*})$ cannot occur. Finally we assume that $v_{p}\notin L(T)$, $v_{q}^{*}\notin L(T^{*})$. Now we select two vertices $v_{p^{\prime}}\in L(T)$ and $v_{q^{\prime}}^{*}\in L(T^{*})$ such that $v_{p^{\prime}}$ is closer to the leaf of the  branching path than $v_p$ in $T$ and $v_{q^{\prime}}^{*}$ is closer to the leaf  of the  branching path than $v_q^*$ in $T^*$.  But then  $(v_{p},v_{q}^{*})\in I_{T\cp T^{*}}((v_{p^{\prime}},v_{j}^{*}),(v_{i},v_{q^{\prime}}^{*}))$, a final contradiction.
\end{proof}

\section{Proof of Theorem~\ref{thm:main}}
\label{sec:proof}

If $T$ and $T^*$ are both paths, then Theorem~\ref{thm:main} holds by Lemma~\ref{p-path}. In the following we may thus without loss of generality assume that $T^*$ is not a path.
Lemma~\ref{lower-gp} implies that $\gp(T\cp T^{*}) \geq \gp(T) + \gp(T^{*})$, hence it  remains to prove that $\gp(T\cp T^{*})\leq \gp(T) + \gp(T^{*})$. Set $n = n(T)$, $n^{*} = n(T^*)$, $V(T)=\{v_{1},\ldots, v_{n}\}$, and $V(T^{*}) = \{v_{1}^{*}, \ldots,v_{n^{*}}^{*}\}$.

Assume on the contrary that there exists a general position set $R$ of $T$ such that $|R| > \gp(T) + \gp(T^{*})$. Since the restriction of $R$ to a $T$-layer of $T\cp T^{*}$ is a general position set of the layer (which is in turn isomorphic to $T$), the restriction contains at most $\gp(T) = \ell(T)$ elements. Similarly, the restriction of $R$ to a $T^*$-layer contains at most $\gp(T^*) = \ell(T^*)$ elements. We now distinguish the following cases.

\medskip\noindent
\textbf{Case 1.} There exists a $T$-layer $T^{v_{j}^{*}}$ with $|V(T^{v_{j}^{*}}) \cap R| = \gp(T)$,  or a $T^*$-layer $\prescript{v_{i}}{}T^{*}$ with $|V(\prescript{v_{i}}{}T^{*})\cap R| = \gp(T^*)$.

By the commutativity of the Cartesian product, we may without loss of generality assume that there is a layer $\prescript{v_{i}}{}{T^{*}}$ with $|R\cap V(\prescript{v_{i}}{}{T^{*}})| = \gp(T^*)$. Let $R = R_1\cup R_2$, where $R_1 = R\cap V(\prescript{v_{i}}{}{T^{*}})$ and $R_2 = R\setminus R_1$, that is, $R_2 = \bigcup\limits_{t\in [n]\setminus\{i\}}\Big(V(\prescript{v_{t}}{}T^{*})\cap R\Big)$. Let further $S^{*}$ be the projection of $R\cap V(\prescript{v_{i}}{}{T^{*}})$ on $T^{*}$, that is, $S^{*}=\{v_{j}^{*}:\ (v_{i},v_{j}^{*})\in R_1\}$. Since $|R_1| = \gp(T^*)$,  our assumption implies $|R_2| \geq \gp(T)+1$. Then, as $\gp(T) = \ell(T)$, there exist two different vertices $w=(v_p,v_q^*)$ and $w^{\prime}=(v_{p^{\prime}},v_{q^{\prime}}^*)$ from $R_2$ such that  $v_p$ and $v_{p^{\prime}}$ lie on a same branching path $P$ of $T$. (Note that it is possible that $v_p = v_{p^{\prime}}$.) We may assume that $d_{T}(v_{p^{\prime}}, x) \le d_{T}(v_p,x)$, where $x$ is the leaf of $P$. We proceed by distinguishing two subcases based on the position of $v_q^*$ and  $v_{q^{\prime}}^*$ in $T^*$.

\medskip\noindent
\textbf{Case 1.1.} There exists a branching path $P^*$ of $T^*$ that contains both $v_q^*$ and $v_{q^{\prime}}^*$. \\
Recall that $T^*$ is not a path. Lemma~\ref{T-non-leaf} implies that a vertex of a tree belongs to a gp-set if and only if it lies on a pendant path and has degree $1$ or $2$. Therefore, we can select $P^*$ with the root of degree at least $3$.  Assume that $d_{T^*}(v_{q^{\prime}}^*, y) \le d_{T^*}(v_q^*,y)$, where $y$ is the leaf of $P^*$. (The reverse case can be treated analogously.)
Since $S^*$ is a gp-set of $T^*$ which is not isomorphic to a path, there is a vertex $v_k^*\in S^*$  lying on $P^*$. So we may consider that $P^*$ is a branching path that contains $v_q^*$, $v_{q^{\prime}}^*$ and a vertex $v_{k}^*\in S^*$. (It is possible that some of these vertices are the same.) Let $z=(v_i,v_k^*)$. Then $z\in R_1$. We proceed by distinguishing the following subcases based on the position of $v_{p}$, $v_{p^{\prime}}$ and $v_{i}$ in $T$.

\medskip\noindent
\textbf{Subcase 1.1.1.} $v_{p^{\prime}}\in I(v_i,v_p)$. \\
In this subcase, if $v_{k}^*$ is closer than $v_q^*$, $v_{q^{\prime}}^*$ to the leaf $y$ of $P^*$, then, by Lemma \ref{geo-interval},  $w^{\prime} \in I_{T\cp T^{*}}(w,z)$, a contradiction.

If $v_{k}^*\in I(v_{q}^*,v_{q^{\prime}}^{*})$, then since $\ell(T^*)\geq 3$, there exists $z^{\prime}=(v_{i},v_{k^{\prime}}^{*})\in \{v_{i}\}\times S^{*}$ such that $v_{k}^{*}$,$v_{q}^{*}\in I(v_{q^{\prime}}^*,v_{k^{\prime}}^{*})$ in $T^*$. Then we have \begin{eqnarray*}
d(w^{\prime},z^{\prime})&=&d_T(v_{p^{\prime}},v_i)+d_{T^*}(v_{q^{\prime}}^*,v_{k^{\prime}}^*)\\
              &=&d_T(v_{p^{\prime}},v_i)+d_{T^*}(v_{q^{\prime}}^*,v_{k}^*)+d_{T^*}(v_{k}^*,v_{k^{\prime}}^*)\\
              &=&d(w^{\prime},z)+d(z,z^{\prime}),
 \end{eqnarray*}
which implies that  $z\in I_{T\cp  T^{*}}(w^{\prime},z^{\prime})$, a contradiction.

\medskip\noindent
\textbf{Subcase 1.1.2.}  $v_{i}\in I(v_{p},v_{p^{\prime}})$.\\
In this subcase, if $v_{k}^*\in I(v_q^*,v_{q^{\prime}}^*)$ in $P^*$, then   $z\in I_{T\cp  T^{*}}(w,w^{\prime})$ by Lemma \ref{geo-interval}, a contradiction.

Assume that $v_{k}^*$ is closer than $v_q^*$, $v_{q^{\prime}}^*$ to the leaf of $P^*$. Since $|S^{*}|=\ell(T^*)\geq  3$, there is a vertex $z^{\prime}=(v_i,v_{k^{\prime}}^*)\in \{v_i\}\times S^{*}$ such that $v_q^*$, $v_{q^{\prime}}^*\in I(v_{k}^*,v_{k^{\prime}}^*)$ in $T^{*}$. Let $v_{k^{\prime}}^*$ be on a branching path ${P^{\prime}}^*$ in $T^{*}$ where ${P^{\prime}}^*\neq P^*$. Note that $\ell(T) + 1\geq3$. There exists at least one vertex $a=(v_{x},v_{y}^{*})\in R_2 \setminus\{w,w^{\prime}\}$. Next we consider the positions of $v_x,v_{y}^{*}$ in $T,T^*$, respectively.

Suppose first that $v_{y}^{*}\in V(P^*\cup {P^{\prime}}^*)$. If $v_x$, $v_p$, $v_{p^{\prime}}$ and $v_i$ lie on a path in $T$, then there are five vertices $w$, $w^{\prime}$, $z$, $z^{\prime}$ and $a$ in $R_2$, three of which lie on a common geodesic in $T\cp T^*$, a contradiction.  Note that if $T$ is a path, then we are done as above. Therefore, assume that $T$ is not isomorphic to a path in the following and the root of $P$ has degree at least $3$. Otherwise, $v_x\notin P$ and $v_x,v_p$ lie on a common branching path in $T$. Let $V_s$ be the set of vertices of $T$ but not contained in $T_{ip^{\prime}}$ where $T_{ip^{\prime}}$ is the subtree of $T-v_p$ containing $v_i$ and $v_{p^{\prime}}$.
If there is a vertex $a^{\prime}=(v_s,v_l^{*})\in R_2$ with $v_s\in V_s$, then $R_2$ contains $w$, $w^{\prime}$, $z$, $z^{\prime}$ and $a^{\prime}$, three of which are on a common geodesic, a contradiction. Therefore, the first coordinate of any vertex in $R_2$ cannot be in $V_s$. Assume that $P^{\prime}\neq P$ is any  branching path  containing $v_p$ and a leaf both in $T_{ip^{\prime}}$ and $T$. Then,  besides $w$, $P^{\prime}\cp T^*$ contains at most one vertex  in $R_2$ of $T\cp T^{*}$. Otherwise, $P^{\prime}\cp T^*$ contain two vertices $h$, $h^{\prime}$ in $R_2$. Then there exist two vertices $h_0,h_{0}^{\prime} \in\{v_i\}\times S^*$ such that three vertices from $\{h,h^{\prime},h_0,h_{0}^{\prime},w\}$ lie on some geodesic in $T\cp T^{*}$, a contradiction. (Here $h_0$ may be equal to $h_{0}^{\prime}$.) Note that $V_s$ contains at least two leaves of $T$ since the root of $P$ (just in $V_s$) has degree at least $3$. Then $T_{ip^{\prime}}$ has at most $\ell(T)-2$ leaves in $T$.  Since $P\cp T^*$ contains two vertices $w$ and $w^{\prime}$ in $R_2$, we have  $|R_2|\leq \ell(T)-2+1< \ell(T)=\gp(T)$, a contradiction with the assumption.

Assume now that $v_{y}^{*}\notin V(P^*\cup {P^{\prime}}^*)$. Then there exists a vertex $z^{\prime\prime}=(v_i,v_{k^{\prime\prime}}^{*})\in\{v_i\}\times S^{*}$ such that $v_{y}^{*},v_{k^{\prime\prime}}^{*}$ lie on a common branching path in $T^*$. If $v_{y}^{*}$ is closer to the leaf of the branching path than $v_{k^{\prime\prime}}^{*}$ in $T^*$, then $v_i\in I(v_x,v_i)$ and $v_{k^{\prime\prime}}^{*}\in I(v_{y}^{*},v_{k}^{*})$. Therefore, by Lemma \ref{geo-interval}, we get $z^{\prime\prime}\in I_{T\cp T^{*}}(a,z)$, a contradiction.
In the case that $v_{k^{\prime\prime}}^{*}$ is closer to the leaf of the branching path than $v_{y}^{*}$ in $T^*$, we consider the positions of $v_x$, $v_p$, $v_{p^{\prime}}$ and $v_i$ in $T$. Let $V_{1}=\{z,z^{\prime},w,w^{\prime},a,z^{\prime\prime}\}$. Then $V_1\subseteq R_2$. If $v_x$, $v_p$, $v_{p^{\prime}}$ and $v_i$ lie on a path in $T$, then there exist three vertices in $V_1$ lying on a common geodesic in $T\cp T^*$, a  contradiction again. Otherwise, $v_x\notin P$ and $v_x,v_p$ lie on a common branching path in $T$. Similarly as above, a contradiction occurs.

\medskip\noindent
\textbf{Subcase 1.1.3.}  $v_{p}\in I(v_{i},v_{p^{\prime}})$.\\
In this subcase, since $\ell(T^*)\geq 3$, there exists a vertex $z^{\prime}=(v_{i},v_{k^{\prime}}^{*})\in \{v_{i}\}\times S^{*}$ such that $v_{k^{\prime}}^{*}\notin P^*$ and $v_{q}^{*}\in I(v_{k^{\prime}}^{*},v_{q^{\prime}}^{*})$ in $T^{*}$. Since \begin{eqnarray*}
d(z^{\prime},w^{\prime})&=&d_T(v_i,v_{p^{\prime}})+d_{T^*}(v_{k^{\prime}}^*,v_{q^{\prime}}^*)\\
              &=&d_T(v_i,v_{p})+d_{T^*}(v_{k^{\prime}}^*,v_{q}^{*})+d_T(v_p,v_{p^{\prime}})+d_{T^*}(v_{q}^*,v_{q^{\prime}}^*)\\
              &=&d(z^{\prime},w)+d(w,w^{\prime}),
 \end{eqnarray*}
we have  $w\in I_{T\cp T^{*}}(z^{\prime},w^{\prime})$, a contradiction.

\medskip\noindent
\textbf{Subcase 1.1.4.} $v_{i}\notin V(P)$ such that $v_{i}$, $v_{p}$ lie on a same branching path in $T$.\\
In this subcase, since $\ell(T^*)\geq 3$, there is a vertex $z^{\prime}=(v_i,v_{k^{\prime}}^*)\in \{v_i\}\times S^{*}$ such that $v_{q}^{*}\in I(v_{k^{\prime}}^{*},v_{k}^{*})$ in $T^{*}$.
If $v_{k}^{*}\in I(v_{q}^{*},v_{q^{\prime}}^*)$ , then obviously $v_{k}^{*}\in I(v_{q}^{*},v_{k^{\prime}}^*)$ and therefore,
\begin{eqnarray*}
d(w^{\prime},z^{\prime})&=&d_T(v_{p^{\prime}},v_i)+d_{T^*}(v_{q^{\prime}}^*,v_{k^{\prime}}^*)\\
              &=&d_T(v_{p^{\prime}},v_i)+d_{T^*}(v_{q^{\prime}}^*,v_{k}^*)+d_{T^*}(v_{k}^*,v_{k^{\prime}}^*)\\
              &=&d(w^{\prime},z)+d(z,z^{\prime})\,.
\end{eqnarray*}
 We conclude that $z\in I_{T \cp T^{*}}(w^{\prime},z^{\prime})$, a contradiction.

If $v_{k}^{*}$ is closer to the leaf of $P^*$ than $v_{q}^{*},v_{q^{\prime}}^*$, then  we get a contradiction similarly as in Subcase 1.1.2.

\medskip\noindent
\textbf{Case 1.2.} $v_q^*$ and $v_{q^{\prime}}^*$ do not lie on a same branching path in $T^*$.\\
In this subcase, we may assume that $v_q^*$ and $v_{q^{\prime}}^*$ lie on distinct branching paths $P^{*}$ and $P^{\prime*}$ in $T^{*}$, respectively. Since $\ell(T^*)\geq 3$ and $T^*$ is not isomorphic to a path, there exist two vertices $z=(v_i,v_{k}^*)$ and $z^{\prime}=(v_{i},v_{k^{\prime}}^{*})$ from $ \{v_i\}\times S^{*}$, such that $v_{k}^*\in P^*$ and $v_{k^{\prime}}^*\in P^{\prime*}$. We consider the following subcases based on the positions of $v_{p}$, $v_{p^{\prime}}$ and $v_{i}$ in $T$.

\medskip\noindent
\textbf{Subcase 1.2.1.}  $v_{p^{\prime}}\in I(v_i,v_p)$.\\
In this subcase, if $v_{k^{\prime}}^*$ is closer than $v_{q^{\prime}}^*$ to the leaf of $P^{\prime*}$, then $v_{p^{\prime}}\in I(v_p,v_i)$ and $v_{q^{\prime}}^*\in I(v_q^*,v_{k^{\prime}}^*)$. Lemma \ref{geo-interval} gives $w^{\prime}\in I_{T\cp T^{*}}(w,z^{\prime})$, a contradiction.  On the other hand, if $v_{q^{\prime}}^*$ is closer than $v_{k^{\prime}}^*$ to the leaf of $P^{\prime*}$, then $v_i\in I(v_i,v_{p^\prime})$ and $v_{k^\prime}^*\in I(v_k^*,v_{q^\prime}^*)$, hence Lemma~\ref{geo-interval} gives  $z^{\prime}\in I_{T\cp T^{*}}(w^{\prime},z)$, a contradiction again.

\medskip\noindent
\textbf{Subcase 1.2.2.} $v_{i}\in I(v_{p},v_{p^{\prime}})$.\\
In this subcase, we first assume that $v_{q^{\prime}}^*$ is closer than $v_{k^\prime}^*$ to the leaf of $P^{\prime*}$.  Then $v_i\in I(v_i,v_{p^{\prime}})$ and $v_{k^\prime}^*\in I(v_k^*,v_{q^\prime}^*)$. Therefore, by Lemma \ref{geo-interval}, we get $z^{\prime}\in I_{T\cp T^{*}}(z,w^{\prime})$ as a contradiction.
Otherwise we suppose that $v_{k^\prime}^*$ is closer than $v_{q^{\prime}}^*$  to the leaf of $P^{\prime*}$.  If $v_{q}^{*}$ is closer than $v_{k}^{*}$ to the leaf of $P^*$, then $v_i\in I(v_p,v_{i})$ and $v_{k}^*\in I(v_{q}^*,v_{k^\prime}^*)$. Therefore, by Lemma \ref{geo-interval}, we get $z\in I_{T\cp T^{*}}(w,z^{\prime})$, a contradiction. In the case that  $v_{k}^*$ is closer than $v_{q}^*$ to the leaf of $P^*$, we find a contradiction similarly as the proof of Subcase 1.1.2.

\medskip\noindent
\textbf{Subcase 1.2.3.} $v_{p}\in I(v_{i},v_{p^{\prime}})$.\\
In this subcase, if $v_{k}^*$ is closer than $v_{q}^*$ to the leaf of $P^{*}$, then $v_{p}\in I(v_{i},v_{p^{\prime}})$ and $v_{q}^*\in I(v_{k}^*,v_{q^{\prime}}^*)$.  So Lemma~\ref{geo-interval} gives $w\in I_{T\cp T^{*}}(z,w^{\prime})$, a contradiction. And if $v_{q}^*$ is closer than $v_{k}^*$ to the leaf of $P^{*}$, then $v_{i}\in I(v_{i},v_{p})$ and $v_{k}^*\in I(v_{k^{\prime}}^*,v_{q}^*)$,  hence we get $z\in I_{T\cp T^{*}}(z^{\prime},w)$.

\medskip\noindent
\textbf{Subcase 1.2.4.} $v_{i}\notin V(P)$ such that $v_{i}$, $v_{p}$ lie on a same branching path in $T$.\\
First suppose that $v_{q}^{*}$ is closer to the leaf than $v_{k}^{*}$ in $P^{*}$, then $v_{i}\in I(v_{i},v_{p})$ and $v_{k}^*\in I(v_{q}^*,v_{k^{\prime}}^*)$. Thus, by Lemma \ref{geo-interval},  we get $z\in I_{T\cp T^{*}}(w,z^{\prime})$.

Assume that $v_{k}^{*}$ is closer than $v_{q}^{*}$ to the leaf of $P^{*}$. If $v_{q^{\prime}}^*$ is closer to the leaf than $v_{k^\prime}^*$, then $v_i\in I(v_i,v_{p^{\prime}})$ and $v_{k^\prime}^*\in I(v_k^*,v_{q^\prime}^*)$, which gives  $z^{\prime}\in I_{T\cp T^{*}}(z,w^{\prime})$. If $v_{{k}^{\prime}}^{*}$ is closer than $v_{{q}^{\prime}}^{*}$ to the leaf of ${P^{\prime}}^*$, we can proceed similarly as in Subcase 1.1.4.

Now we turn to the second case.

\medskip\noindent
\textbf{Case 2.} $|R\cap V(\prescript{v_{k}}{}{T^{*}})| < \ell(T^*)$ for any $k\in [n]$, and $|R\cap V(T^{v_t^*})| < \ell(T)$ for any $t\in [n^*]$.\\
In this case, let $\prescript{v_{i}}{}{T^{*}}$ be a layer with $|R\cap V(\prescript{v_{i}}{}{T^{*}})| = \max\{|R\cap V(\prescript{v_{k}}{}{T^{*}})|:k\in[n]\}$. Let $R = R_1\cup R_2$ where $R_1 = R\cap V(\prescript{v_{i}}{}{T^{*}})$ and $R_2 = R\setminus R_1$, that is, $R_2 = \bigcup\limits_{k\in [n]\setminus\{i\}}\Big(V(\prescript{v_{k}}{}T^{*})\cap R\Big)$. Set further $S^{*}=\{v_{j}^{*}:\ (v_{i},v_{j}^{*})\in R_1\}$. Then $1\leq|S^{*}|\leq \ell(T^*)-1$.

Assume first $|S^*|=1$. Therefore $|R\cap V(\prescript{v_{k}}{}{T^{*}})|\leq1$ for any $k\in[n]$. Next we only need to consider $|R\cap V(T^{v_j^*})|\leq1$ for any $j\in[n^*]$. (If $|R\cap V(T^{v_j^*})|\geq2$ for some $j\in[n^*]$, by commutativity of $T\cp T^*$, the proof is similar to the subcase  in which $2\leq|S^{*}|\leq\ell(T^*)-1$.) Therefore, suppose that $|R\cap V(T^{v_j^*})|\leq1$ for any $j\in[n^*]$. Then $|R|\leq\min\{n,n^{*}\}$.  We now claim that $|R|\leq \ell(T)+\ell(T^*)$. If not, then  since $|R|\geq \ell(T)+\ell(T^*)+1\geq  6$, there exist three vertices $u=(v_{p},v_{j}^{*})$, $v=(v_{p^{\prime}},v_{q}^{*})$ and $w=(v_{s},v_{\ell}^{*})$  from $R$ such that $v_p,v_{p^{\prime}}$ lie on a same branching path in $T$, and $v_{j}^{*},v_{\ell}^{*}$ lie on a common branching path in $T^*$. Note that there may be $p^{\prime} = s, q=\ell$. But we can always select  a vertex $h\in R \setminus \{u,v,w\}$ such that $u,v,h$ or $u,w,h$ lie on a same geodesic in $T\cp T^{*}$, which is a contradiction. So our result holds  when  $|S^*|=1$.

Suppose  second that $2\leq|S^{*}|\leq\ell(T^*)-1$.  As $|R_1| = |S^*|$, we need to prove that $|R_2|\leq \ell(T)+\ell(T^*)-|S^*|$. Assume on the contrary that $|R_2|\geq \ell(T)+\ell(T^*)-|S^*| + 1$.
Since $|S^*|\geq 2$, there are two distinct vertices $w=(v_{i},v_{j}^{*})$ and $w^{\prime}=(v_{i},v_{j^{\prime}}^{*})$ from $\{v_{i}\}\times S^{*}$. We distinguish the following cases based on the positions of $v_j^*$, $v_{j^{\prime}}^*$ in $T^*$.

\medskip\noindent
\textbf{Case 2.1.} $v_j^*$ and $v_{j^{\prime}}^*$ lie on a same branching path $P^*$ of $T^*$.\\
In this subcase,  we may without loss of generality assume that $v_{j^{\prime}}^*$ is closer than $v_j^*$  to the leaf of $P^*$.  Let $T^*_{v_{j^{\prime}}^*}$ be the maximal subtree of $T^*-v_j^*$ containing $v_{j^{\prime}}^*$ and let $V_{s^*} = V(T^*)\setminus V(T^*_{v_{j^{\prime}}^*})$. Let further $S_{1}^{*} = \{v_{q}^{*}:\ v_{q}^{*}\in I(v_j^*,v_{\ell}^*), v_\ell^*\in S^{*}\cap V(T^*_{v_{j^{\prime}}^*})\}$. Now we prove the following claim.

\medskip\noindent
\textbf{Claim 1.} If $z=(v_{p},v_{t}^{*})\in R_2$, then $v_{t}^{*}\in S_{1}^{*}$ .

\medskip\noindent
\textbf{Proof of Claim 1.}
If not, suppose first that $v_{t}^{*}\in V(P^*)$ is closer than $v_{j^{\prime}}^*$ to the leaf of $P^*$. Then $v_i\in I(v_i,v_p)$ and $v_{j^{\prime}}^*\in I(v_{t}^{*},v_{j}^*)$. Hence,  $w^{\prime}\in I_{T\cp T^{*}}(w,z)$. And if
$v_{t}^{*}\in V_{s^*}$, then $v_j^*\in I(v_{t}^{*},v_{j^{\prime}}^*)$. Combining this fact with $v_{i}\in I(v_{i},v_{p})$, we have  $w\in I_{T\cp T^{*}}({w^\prime},z)$. This proves Claim 1.

By Claim 1,
 we have $|\bigcup\limits_{v_{t}^{*}\in S_{1}^{*}}\big(V(T^{v_{t}^{*}})\cap R\big)| \geq \ell(T)+\ell(T^{*})-|S^{*}| + 1 \geq \ell(T)+1$.
Then there exist two vertices $z=(v_{p},v_{\ell}^{*})$ and $z^{\prime}=(v_{p^{\prime}},v_{\ell^{\prime}}^{*})$ from $\cup_{v_{t}^{*}\in S_{1}^{*}}\big(V(T^{v_{t}^{*}})\cap R\big)$ such that $v_{\ell}^{*},v_{\ell^{\prime}}^{*}\in S_{1}^{*}$ and $v_{p},v_{p^{\prime}}$ lie on a same branching path $P$ in $T$. Without loss of generality, let $v_{p^{\prime}}$ be closer than $v_{p}$ to the leaf of $P$, and let $v_{\ell}^{*},v_{\ell^{\prime}}^{*}\in I(v_j^*,v_{j^{\prime}}^*)$ (by the definition of $S_1^*$). We consider the following subcases according to the positions of $v_{i},v_{p},v_{p^{\prime}}$ in $T$.

\medskip\noindent
\textbf{Subcase 2.1.1.} $v_{p^{\prime}}\in I(v_i,v_p)$.\\
If $v_{\ell^{\prime}}^{*}$ is closer than $v_{\ell}^{*}$ to $v_{j^{\prime}}^*$ in $P^{*}$, then we have $v_{p^{\prime}}\in I(v_{i},v_{p})$ and $v_{\ell^{\prime}}^*\in I(v_{\ell}^*,v_{{j^\prime}}^*)$. Therefore, $z^{\prime}\in I_{T\cp T^{*}}(z,{w^\prime})$. And if $v_{\ell}^{*}$ is closer than $v_{\ell^{\prime}}^{*}$ to $v_{j^{\prime}}^*$ in $P^{*}$, then we have $v_{p^{\prime}}\in I(v_{i},v_{p})$ and $v_{\ell^{\prime}}^*\in I(v_{\ell}^*,v_{j}^*)$ and so $z^{\prime}\in I_{T\cp T^{*}}(z,{w})$.

\medskip\noindent
\textbf{Subcase 2.1.2.} $v_{i}\in I(v_{p},v_{p^{\prime}})$.\\
Note that $\ell(T)+\ell(T^{*})-|S^{*}| + 1 \geq 4$. Then there  exists at least a vertex $a=(v_{x},v_{y}^{*})\in \cup_{v_{t}^{*}\in S_{1}^{*}}\big(V(T^{v_{t}^{*}})\cap R\big)$ different from $z$ and $z^{\prime}$. Based on the position of $v_y^*$ ($v_{y}^{*}\in P^*$ or $v_{y}^{*}\notin P^*$) in $T^*$, and the positions of $v_x$, $v_{i}$, $v_{p}$ and $v_{p^{\prime}}$ in $T$, we get contradictions using a similar proof as in Subcase 1.1.2.

\medskip\noindent
\textbf{Subcase 2.1.3.}  $v_{p}\in I(v_{i},v_{p^{\prime}})$.\\
If $v_{\ell^{\prime}}^{*}$ is closer than $v_{\ell}^{*}$  to $v_{j^{\prime}}^*$ in $T^{*}$, then $v_{p}\in I(v_{i},v_{p^{\prime}})$ and $v_{\ell}^*\in I(v_{j}^*,v_{\ell^\prime}^*)$, therefore $z\in I_{T\cp T^{*}}(w,{z^\prime})$. And if  $v_{\ell}^{*}$ is closer than $v_{\ell^{\prime}}^{*}$ to $v_{j^{\prime}}^*$  in $T^{*}$, then $v_{p}\in I(v_{i},v_{p^{\prime}})$ and $v_{\ell}^*\in I(v_{j^{\prime}}^*,v_{\ell^\prime}^*)$, hence $z\in I_{T\cp T^{*}}(w,{z^\prime})$.

\medskip\noindent
\textbf{Subcase 2.1.4.} $v_{i}\notin V(P)$ such that $v_{i}$, $v_{p}$ lie on a same branching path in $T$.\\
Since $\ell(T)+\ell(T^{*})-|S^{*}| + 1 \geq 4$, there exists a vertex $(v_{x},v_{y}^{*})\in \cup_{v_{t}^{*}\in S_{1}^{*}}\big(V(T^{v_{t}^{*}})\cap R\big)$. Proceeding similarly as in Subcase 1.1.4, we get required contradictions.  But then $|\cup_{v_{t}^{*}\in S_{1}^{*}}\big(V(T^{v_{t}^{*}})\cap R\big)| \leq \ell(T)+\ell(T^{*})-|S^{*}|$, a contradiction with the assumption.

\medskip\noindent
\textbf{Case 2.2.} $v_j^*$,$v_{j^{\prime}}^*$ lie on different branching paths $P^*$, $P^{\prime*}$ in $T^*$, respectively.\\
In this subcase, let $S_2^*$ be a set of vertices of $\prescript{v_{i}}{}T^{*}$ closer to the leaf of a branching path than $v_{g}^{*}$ for any $v_{g}^{*}\in S^*$. Note that  $S^*\cap S_{2}^{*} = \emptyset$. We prove the following claim.

\medskip\noindent
\textbf{Claim 2.} If $(v_{p},v_t^{*})$ in  $R_2$, then $v_{t}^{*}\in V(T^{*})\setminus(S^*\cup S_{2}^{*})$.

\medskip\noindent
\textbf{Proof of Claim 2.}
Lemma~\ref{lem:lay-v} implies $v_{t}^{*}\notin S^*$.
Assume that $v_{t}^{*}\in S_{2}^*$ lies on a same branching path for some $v_{g}^{*}$ in $T^{*}$.  Note that $|S^*| \geq 2$. Then there exists another vertex $v_{g^{\prime}}^*$ such that $v_g^*\in I(v_t^*,v_{g^{\prime}}^*)$. Combining this fact with $v_{i}\in I(v_{i},v_{p})$, we arrive at a contradiction $w\in I_{T\cp T^{*}}(z,{w^\prime})$. This proves Claim 2.

Let now $S_{1^\prime}^* = \{v_q^*:\ v_q^*\in I(v_g^*,v_{g^{\prime}}^*),v_g^*,v_{g^{\prime}}^*\in S^*\}$. By a parallel reasoning as in Subcase 2.1 and with Claim~2 in hands we infer that $|\cup_{v_{t}^{*}\in S_{1^{\prime}}^{*}}\big(V(T^{v_{t}^{*}})\cap R\big)|\leq \ell(T)$.

Let $S = \{v_{k}:\ (v_{k},v_{t}^{*})\in\cup_{v_{t}^{*}\in S_{1^{\prime}}^{*}}\big(V(T^{v_{t}^{*}})\cap R\big)\}$  and set $S^{**} = V(T^{*})\setminus (S^{*}\cup S_{1^\prime}^{*})$. From the assumption we have
$|\cup_{v_{t}^{*}\in  S^{**}}\big(V(T^{v_{t}^{*}})\cap R\big)| \geq \ell(T)+\ell(T^*)-|S|-|S^*| +1$.
So there exists a vertex $z=(v_p,v_{\ell}^{*})\in\cup_{v_{t}^{*}\in  S^{**}}\big(V(T^{v_{t}^{*}})\cap R\big)$, and we can always select two distinct vertices $u=(v_h,v_{g}^*)$ and $v=(v_{h^\prime},v_{g^\prime}^*)$ from  $R$ such that $v_p$ and $v_h$ lie on a same branching path in $T$, while  $v_{\ell}^{*}$ and $v_{g^\prime}^*$ lie on a common branching path in $T^*$. But we can choose another vertex $w\in  R$ such that either $u,w,z$ or $u,v,z$ lie on a same geodesic in $T\cp T^{*}$ as a contradiction. Therefore,
$$|\bigcup\limits_{v_{t}^{*}\in  S^{**}}\Big(V(T^{v_{t}^{*}})\cap R\Big)|  \leq \ell(T)+\ell(T^*)-|S|-|S^*|.$$
and we are done.

\section*{Acknowledgements}
Kexiang Xu is supported by NNSF of China (grant No.\ 11671202, and the China-Slovene bilateral grant 12-9). Sandi Klav\v{z}ar acknowledges the financial support from the Slovenian Research Agency (research core funding P1-0297, projects J1-9109, J1-1693, N1-0095, and the bilateral grant BI-CN-18-20-008).

\end{document}